\documentclass[12pt,amstex]{amsart}

\usepackage{epsfig}
\usepackage{amsmath}
\usepackage{amssymb}
\usepackage{amscd}
\usepackage{graphicx}
\usepackage[all]{xy}
\usepackage{pstricks}
\usepackage{pst-node}
\usepackage[english]{babel}
\usepackage{float}
\usepackage{enumerate}

\newtheorem{thm}{Theorem}[section]
\newtheorem{lem}[thm]{Lemma}

\newtheorem{ques}[thm]{Question}
\newtheorem{cor}[thm]{Corollary}

\theoremstyle{definition}
\newtheorem{de}[thm]{Definition}

\theoremstyle{remark}

\numberwithin{equation}{section}

\input xy
\xyoption{all}

\def \N {\mathbb N}

\def \Z {\mathbb Z}

\def \Q {{\bf Q}}
\def \RP {{\bf RP}}

\def \id {{\rm id}}

\def \e {\epsilon}

\begin{document}
\title{Enveloping semigroups of systems of order d}

\author{Sebasti\'an Donoso}
\address{Centro de Modelamiento Matem\'atico and Departamento de Ingenier\'{\i}a
Matem\'atica, Universidad de Chile, Av. Blanco Encalada 2120,
Santiago, Chile. \newline  Universit\'e Paris-Est, Laboratoire d'analyse et de math\'ematiques
appliqu\'ees, 5 bd Descartes, 77454 Marne la Vall\'ee
Cedex 2, France} \email{sdonoso@dim.uchile.cl, sebastian.donoso@univ-paris-est.fr }

\thanks{ The author was supported by CONICYT doctoral fellowship, BCH-Cotutela doctoral fellowship and CMM-Basal grants.}

\begin{abstract}
In this paper we study the Ellis semigroup of a $d$-step nilsystem and the inverse limit of such systems. By using the machinery of cubes developed by Host, Kra and Maass, we
prove that such a system has a $d$-step topologically nilpotent enveloping semigroup. In the case $d=2$, we prove that these notions are equivalent, extending a previous result by Glasner.

\end{abstract}

\maketitle

\section{Introduction}

A topological dynamical system is a pair $(X,T)$, where $X$ is a compact metric space and $T\colon X\to X$ is a homeomorphism. Several aspects of the dynamics of $(X,T)$ can be deduced from algebraic properties of its enveloping semigroup $E(X)$ (introduced by Ellis in the 1960's; see Section 2.6 for definitions). In particular, a topological dynamical system is a rotation on a compact abelian group if and only if its enveloping semigroup is an abelian group. Other interesting applications can be found in \cite{Aus}, \cite{Ellis} and \cite{GlasRev}.

In recent years the study of the dynamics of rotations on nilmanifolds and inverse limits of this kind of dynamics has drawn much interest. In particular, we point to the applications in ergodic theory \cite{HK05}, number theory and additive combinatorics (see for example \cite{GT}).

Following terminology introduced in \cite{HK05}, let us call a minimal topological dynamical system {\it a system of order $d$} if it is either a $d$-step nilsystem or an inverse limit of $d$-step nilsystems. It is revealed in \cite{HKM} that they are a natural generalization of rotations on compact abelian groups and they play an important role in the structural analysis of topological dynamical systems. Particularly, systems of order 2 are the correct framework to study Conze-Lesigne algebras \cite{HK05}.

In this work we are interested in algebraic properties of the enveloping semigroup of a system of order $d$. A first question one can ask is if an enveloping semigroup is a $d$-step nilpotent group. Secondly, a deeper one : Does the property of having an enveloping semigroup that is a $d$-step nilpotent group characterize systems of order $d$?

Even when $E(X)$ is a compact group, multiplication needs not to be a continuous operation. For this reason we introduce the notion of topologically nilpotent group, which is a stronger condition than algebraically nilpotent, and it is more convenient to establish a characterization of systems of order $d$ (see Section 2.4 for definitions). 

Using the machinery of cubes developed by Host, Kra and Maass \cite{HKM}, we prove:

\begin{thm} \label{thm1}
 Let $(X,T)$ be a system of order $d$. Then, its enveloping semigroup is a $d$-step topologically nilpotent group and thus it is a $d$-step nilpotent group.
\end{thm}

Let $A$ be an integer unipotent matrix (this means that $(A-I)^k=0$ for some $k\in \N$) and let $\alpha\in \mathbb{T}^d$. Let $X=\mathbb{T}^d$ and consider the transformation $Tx=Ax+\alpha$. The topological dynamical system $(X,T)$ is an affine $d$-step nilsystem. In \cite{Pikula} it was proved that affine $d$-step nilsystems have nilpotent enveloping semigroups, and an explicit description of those semigroups was given. Theorem \ref{thm1} generalizes this for more general systems, though does not give the explicit form of the enveloping semigroup.

\medskip

The second question is more involved and has been tackled before by Glasner in \cite{Glas}. There, in the case $d=2$, he proved that when $(X,T)$ is an extension of its maximal equicontinuous factor by a torus $K$, the following are equivalent:
  \begin{enumerate}
   \item $E(X)$ is a 2-step nilpotent group
    \item There exists a 2-step nilpotent Polish group $G$ of continuous transformations of $X$, acting transitively on $X$ and there exists   a closed cocompact subgroup $\Gamma\subseteq G$ such that: (i) $T\in G$, (ii) $K$ is central in $G$, (iii) $[G,G]\subseteq K$ and the homogeneous space $(G/\Gamma,T)$ is conjugate to $(X,T)$.
  \end{enumerate}

The assumption that $K$ is a torus can be removed, but one only obtain an extension of the system $(X,T)$ where condition (2) is satisfied.

We proved that systems satisfying condition (2) are actually systems of order 2 (but not every system of order 2 needs to satisfy condition (2)). More generally we prove:

   \begin{thm} \label{thm2}
    Let $(X,T)$ be a minimal topological dynamical system. Then the following are equivalent:
      \begin{enumerate}
       \item $(X,T)$ is a system of order 2.
        \item $E(X)$ is a 2-step topologically nilpotent group.
        \item $E(X)$ is a 2-step nilpotent group and $(X,T)$ is a group extension of an equicontinuous system.
        \item $E(X)$ is a 2-step nilpotent group and $(X,T)$ is an isometric extension of an equicontinuous system.
        \end{enumerate}
   \end{thm}

We do not know if the condition of having a 2-step nilpotent enveloping semigroup by itself is enough to guarantee that $(X,T)$ is a system is of order 2. 


\medskip

The natural question that arises from this result is the converse of Theorem \ref{thm1} in general:
 \begin{ques} 
  Let $(X,T)$ be a system with a $d$-step topologically nilpotent enveloping semigroup with $d>2$. Is $(X,T)$ a system of order d?
 \end{ques}

{\bf Acknowledgement.} I thank Bryna Kra and my advisors Alejandro Maass and Bernard Host for their help and useful discussions.

\section{Preliminaries}
\subsection{Topological dynamical systems}
A {\em topological dynamical system} (or just a {\it system}) is a pair $(X, T)$, where $X$ is a compact metric space and $T\colon X
\rightarrow  X$ is a
homeomorphism of X to itself. We let $d(\cdot,\cdot)$ denote the metric in $X$.

A topological dynamical system $(X, T)$ is transitive if there exists a point $x\in X$
whose orbit $\{T^nx: n\in \Z\}$ is dense in $X$. The system is {\em minimal} if the
orbit of any point is dense in $X$. This property is equivalent to
saying that X and the empty set are the only closed invariant sets
in $X$.

\subsection{Factors and extensions}
A {\em homomorphism} $\pi: X\rightarrow Y$ between the topological dynamical systems
$(X,T)$ and $(Y,S)$ is a continuous onto map such that $\pi\circ T=S\circ \pi $; one says that $(Y,S)$ is a {\it factor} of $(X,T)$ and that
$(X,T)$ is an {\em extension} of $(Y,S)$, and one also refers to
$\pi$ as a {\em factor map} or an {\em extension}.
The systems are
said to be {\em conjugate} if $\pi$ is bijective.

\medskip

Let $(X,T)$ be a topological dynamical system and suppose that we have a compact group $U$ of homeomorphism of $X$ commuting with $T$ (where $U$ is endowed with the topology of uniform convergence). The quotient space  $Y=X\backslash U=\{Ux: x \in X\}$ is a metric compact space and if we endow it with the action induced by $T$ we get a topological dynamical system. By definition, the quotient map from $X$ to $Y$ defines a factor map. We say that $(X,T)$ is an extension of $(Y,T)$ by the group $U$.

\medskip

Let $(X,T)$ and $(Y,T)$ be minimal topological dynamical systems and let $\pi:X\to Y$ be a factor map. We say that $(X,T)$ is an isometric extension of $(Y,T)$ if for every $y\in Y$ there exists a metric $d_y$ in $\pi^{-1}(y)\times \pi^{-1}(y)$ with the following properties: \begin{enumerate}[(i)]                                                                                                                                                                                                                                                                                    \item (Isometry) If $x,x'\in \pi^{-1}(y)$ then $d_{y}(x,x')=d_{Ty}(Tx,Tx')$. \item (Compatibility of the metrics) If $(x_n,x_n')\in \pi^{-1}(y_n)$ and $(x_n,x_n')\to (x,x')\in \pi^{-1}(y)$ then  $d_{y_n}(x_n,x_n')\to d_{y}(x,x')$.
                                                                                                                                                                                                                                                                     \end{enumerate}

\subsection{Cubes and dynamical cubes}

Let $d\ge 1$ be an integer, and write $[d] = \{1,
2,\ldots , d\}$. We view an element of $\{0, 1\}^d$, the Euclidean cube, either as
a sequence $\e=(\e_1,\ldots,\e_d)$ of $0'$s and $1'$s; or as a
subset of $[d]$. A subset $\e$ corresponds to the sequence
$(\e_1,\ldots, \e_d)\in \{0,1\}^d$ such that $i\in \e$ if and
only if $\e_i = 1$ for $i\in [d]$.  For example, $\vec{
0}=(0,\ldots,0)\in \{0,1\}^d$ is the same as $\emptyset \subset
[d]$ and $\vec{1}=(1,\ldots,1)$ is the same as $[d]$.

If $\vec{n}=(n_1,\ldots,n_d)\in \Z^d$ and $\e\in\{0,1\}^d$, we define $\vec{n}\cdot \e=\sum\limits_{i=1}^n n_i\cdot \e_i=\sum\limits_{i\in \e} n_i $.

If $X$ is a set, we denote $X^{2^d}$ by $X^{[d]}$ and we write a point ${\bf x}\in X^{[d]}$ as ${\bf x}=(x_{\e}: \e\in \{0,1\}^d)$.

\begin{de} 
 Let $(X,T)$ be a topological dynamical system and $d$ an integer.  We define $\Q^{[d]}(X)$ to be the closure in $X^{[d]}=X^{2^d}$ of the elements of the form  
$$  (T^{\vec{n}\cdot \e}x:\e=(\e_1,\ldots,\e_d) \in \{0,1\}^d)$$
where $\vec{n}=(n_1,\ldots,n_d)\in \Z^d$ and $x\in X$.
\end{de}
As examples, $\Q^{[2]}(X)$ is the closure in $X^{[2]}$ of the set
$$\{(x, T^mx, T^nx, T^{n+m}x) : x \in X, m, n\in\Z\}$$ and $\Q^{[3]}(X)$
is the closure in $X^{[3]}$ of the set $$\{(x, T^mx, T^nx,
T^{m+n}x, T^px, T^{m+p}x, T^{n+p}x, T^{m+n+p}x) : x\in X, m, n, p\in
\Z\}.$$

An element in $\Q^{[d]}(X)$ is called a {\it cube of dimension $d$}. The cube structure of a dynamical system was introduced in \cite{HKM} as the topological counterpart of the theory of ergodic cubes developed in \cite{HK05}.

\subsection{Nilpotent groups, nilmanifolds and nilsystems}

Let $G$ be a group. For $g, h\in G$, we write $[g, h] =
ghg^{-1}h^{-1}$ for the commutator of $g$ and $h$ and for $A,B\subseteq G$ we write
$[A,B]$ for the subgroup spanned by $\{[a, b] : a \in A, b\in B\}$.
The commutator subgroups $G_j$, $j\ge 1$, are defined inductively by
setting $G_1 = G$ and $G_{j+1} = [G_j ,G]$. Let $d \ge 1$ be an
integer. We say that $G$ is {$d$-step nilpotent} if $G_{d+1}$ is
the trivial subgroup.

Since we work with groups which are also topological spaces, we can also consider a topological definition of nilpotent which is more suitable for our purposes. Let $G$ be a topological space with a group structure. For $A,B\subseteq G$, we define $[A,B]_{\text{top}}$ as the {\it closed subgroup} spanned by $\{[a, b] : a \in A, b\in B\}$. The topological commutators subgroups $G_j^{\text{top}}$, $j\ge 1$, are defined by setting ${G}_{1}^{\text{top}}=G$ and ${G}_{j+1}^{\text{top}}=[{G}_j^{\text{top}},G]_{\text{top}}$. Let $d \ge 1$ be an integer. We say that $G$ is {\em $d$-step topologically nilpotent} if ${G}_{d+1}^{\text{top}}$ is
the trivial subgroup.

Since $G_j\subseteq {G}_j^{\text{top}}$ for every $j\geq 1$, we have that if $G$ is $d$-step topologically nilpotent, then $G$ is also $d$-step nilpotent. In this sense Theorem \ref{thm1} has stronger conclusions than the previous known particular cases.



\medskip

Let $G$ be a $d$-step nilpotent Lie group and $\Gamma$ a discrete
cocompact subgroup of $G$. The compact manifold $X = G/\Gamma$ is
called a {\em $d$-step nilmanifold}. We recall here that since $G$ is a nilpotent Lie group, the commutators subgroups are closed and then, in this case the notions of $d$-step nilpotent and $d$-step topologically nilpotent coincide (see \cite{Malcev}).

The group $G$ acts on $X$ by
left translations and we write this action as $(g, x)\mapsto gx$.

Let $\tau\in G$ and $T$ be the transformation $x\mapsto \tau x$. Then $(X, T)$ is
called a {\em $d$-step nilsystem}. 

\medskip

The following theorem relates the notion of cubes and nilsystems, and is the main tool used in this article. We recall that a system of order $d$ is either a $d$-step nilsystem or an inverse limit of $d$-step nilsystems.

\medskip

\begin{thm}[\cite{HKM}]\label{HKM}
Assume that $(X, T)$ is a transitive topological dynamical system
and let $d \ge 1$ be an integer. The following properties are
equivalent:
\begin{enumerate}
  \item If $x, y \in X$ are such that $(x, y,\ldots , y) \in  \Q^{[d+1]}(X)$,
then $x = y$.
  \item $X$ is a system of order d. 
\end{enumerate}
\end{thm}

\subsection{Proximality and regionally proximality of order d}

Let $(X,T)$ be a topological dynamical system and $(x,y)\in X\times X$. We say that $(x,y)$ is a {\it proximal}
pair if
\begin{equation*}
\inf_n d (T^nx,T^ny)=0,
\end{equation*}
and it is a {\em distal} pair if it is not proximal. A topological dynamical system
$(X,T)$ is called {\em distal} if $(x,y)$ is distal whenever
$x,y\in X$ are distinct.

\medskip

Let $(X, T)$ be a topological dynamical system and let $d\ge 1$ be an integer. A pair $(x,
y) \in X\times X$ is said to be {\em regionally proximal of order
$d$} if $(x,y,\ldots,y)\in\Q^{[d+1]}(X)$.

The set of regionally proximal pairs
of order $d$ is denoted by $\RP^{[d]}(X)$, and is called {\em the regionally proximal relation
of order $d$}. Thus, $(X,T)$ is a system of order $d$ if and only if the regionally proximal relation of order $d$ coincides with the diagonal relation.

\medskip
The following theorem shows some properties of the regionally proximal relation of order $d$.

\begin{thm}[\cite{HKM}, \cite{SY}]\label{thm-1}
Let $(X, T)$ be a minimal topological dynamical system and $d\in \N$. Then
\begin{enumerate}
\item $(x,y)\in \RP^{[d]}(X)$ if and only if there exists a sequence $(\vec{n}_i)$ in $\Z^{d+1}$ such that $T^{\vec{n}_i\cdot \e} x\to y$ for every $\e\neq \emptyset$.

\item $\RP^{[d]}(X)$ is an equivalence relation.

\item Let $\pi: (X,T)\rightarrow (Y,T)$ be a factor map of minimal systems
and $d\in \N$. Then $\pi\times \pi (\RP^{[d]}(X))=\RP^{[d]}(Y)$.

\end{enumerate}

Furthermore the quotient of $X$ under $\RP^{[d]}(X)$ is the maximal
$d$-step nilsystem and we denote $X/\RP^{[d]}(X)=Z_d(X)$. Particularly $Z_1(X)$ is the maximal equicontinuous factor. It also follows that every factor of a system of order $d$ is a system of order $d$.
\end{thm}

\subsection{The Enveloping semigroup}

The {\it enveloping semigroup} (or {\it Ellis semigroup}) $E(X)$ of a system $(X,T)$ is
defined as the closure in $X^X$ of the set $\{T^n: n\in \Z\}$ endowed with the product topology. For an enveloping semigroup $E(X)$, the applications $E(X)\to E(X)$, $p\mapsto pq$ and $p\mapsto Tp$ are continuous for all $q\in E(X)$.

 This notion was introduced by Ellis and has proved to be a useful tool in studying dynamical systems. Algebraic properties of $E(X)$ can be translated into dynamical properties of $(X,T)$ and vice versa. To illustrate this fact, we recall the following theorem.

 \begin{thm}[see \cite{Aus}, Chapters 3,4 and 5]
 Let $(X,T)$ be a topological dynamical system. Then 
   \begin{enumerate}
    \item $E(X)$ is a group if and only if (X,T) is distal. 
    \item $E(X)$ is an abelian group if and only if $(X,T)$ is equicontinuous if and only if $E(X)$ is a group of continuous transformations.
   \end{enumerate}
  \end{thm}

Particularly, $E(X)$ is a topological group only in the equicontinuous case.

\medskip

 For a distal system, we let $(E_j^{\text{top}}(X))_{j\in\N}$ denote the sequence of topological commutators of $E(X)$.

Let $(X,T)$ and $(Y,T)$ be topological dynamical systems and $\pi:X\to Y$ a factor map. Then there is a unique continuous semigroup homomorphism $\pi^{\ast}:E(X)\to E(Y)$ such that $\pi(ux)=\pi^{\ast}(u)\pi(x)$ for all $x\in X$ and $u\in E(X)$.

Note that if $\pi:X\to Y$ is a factor map between distal systems, we have that $\pi^{\ast}(E_j^{\text{top}}(X))=E_j^{\text{top}}(Y)$ for every $j \geq 1$.

\section{Enveloping semigroups of systems of order d}

In this section we prove Theorem \ref{thm1}. We introduce some notation.

Let $d\geq 1$ be an integer. For $0\leq j\leq d$, let $J\subset [d]$, with cardinality $d-j$ and let $\eta\in \{0,1\}^J$. The subset $$\alpha=\{\e \in \{0,1\}^d : \e_i=\eta_i \text{ for every }i \in J \}\subseteq \{0,1\}^d$$
is called a {\it face of dimension $j$} or equivalently, {\it a face of codimension $d-j$}. 

Given $u\colon X\to X$, $d\in \N$ and $\alpha\subseteq \{0,1\}^d$ a face of a given dimension, we define $u^{[d]}_{\alpha}\colon X^{[d]}\to X^{[d]}$ as $$u^{[d]}_{\alpha}{\bf x}=
\left\{
  \begin{array}{ll}
    (u^{[d]}_{\alpha}{\bf x})_\e=ux_\e, & \hbox{$ \e \in \alpha;$} \\
    (u^{[d]}_{\alpha}{\bf x})_\e=x_\e, & \hbox{$ \e \not \in \alpha .$}
  \end{array}
\right.$$

\medskip

Our theorem follows from the following lemma.

\begin{lem} \label{cubes}
Let $(X,T)$ be a distal topological dynamical system and let $E(X)$ be its enveloping semigroup. Then, for every $d,j\in \N$ with $j\leq d$ and $u\in E_j^{\text{top}}(X)$, we have that $\Q^{[d]}(X)$ is invariant under $u_{\alpha}^{[d]}$ for every face $\alpha$ of codimension $j$.
\end{lem}
\begin{proof}
 Let $d\in \N$. Let $u\in E(X)$ and let $(n_i)$ be a sequence in $\Z$ with $T^{n_i}\to u$ pointwise. Let $\alpha$ be a face of codimension 1. Since $\Q^{[d]}(X)$ is invariant under $T_{\alpha}^{[d]}$, it is also invariant under $(T_{\alpha}^{n_i})^{[d]}$ for every $i \in \Z$. Since $\Q^{[d]}(X)$ is closed and $(T_{\alpha}^{n_i})^{[d]}\to u_{\alpha}^{[d]}$ we get that $\Q^{[d]}(X)$ is invariant under $u_{\alpha}^{[d]}$. Let $1<j\leq d$ and  suppose that the statement is true for every $i<j$. Let $\alpha$ be a face of codimension $j$. We can see $\alpha$ as the intersection of a face $\beta$ of codimension $j-1$ and a face $\gamma$ of codimension 1. Let $u_{j-1}\in E_{j-1}^{\text{top}}(X)$ and $v\in E(X)$ and remark that  $[u_{j-1},v]_{\alpha}^{[d]}=[(u_{j-1})_{\beta}^{[d]},v_{\gamma}^{[d]}]$. Since $(u_{j-1})_{\beta}^{[d]}$ and $v_{\gamma}^{[d]}$ leave invariant $\Q^{[d]}(X)$, so does $[u_{j-1},v]_{\alpha}^{[d]}$. 
 
 As $\Q^{[d]}(X)$ is closed, $E_{\alpha}=\{ u\in E(X): u^{[d]}_{\alpha} \text{ leaves invariant } \Q^{[d]}(X)\}$ is a closed subgroup of $E(X)$ and contains the elements of the form $[u_{j-1},v]$ for $u_{j-1}\in E_{j-1}^{\text{top}}(X)$, $v\in E(X)$. We conclude that $E_j^{\text{top}}(X)\subseteq E_{\alpha}$, completing the proof.
 
 \end{proof}

We use this to prove Theorem \ref{thm1}:

\begin{proof}[ Proof of Theorem \ref{thm1}]
 Let $(X,T)$ be a system of order $d$. Recall that $E(X)$ is a group since $(X,T)$ is a distal system. Let $u\in E_{d+1}^{\text{top}}(X)$ and $x\in X$. By Lemma \ref{cubes} we have that $(x,\ldots,x,u x)\in \Q^{[d+1]}(X)$ and by Theorem \ref{HKM} we have that $ux=x$. Since $x$ and $u$ are arbitrary, we conclude that $E_{d+1}^{\text{top}}(X)$ is the trivial subgroup.

\end{proof}

\section{Proof of Theorem \ref{thm2}   }

We start with some lemmas derived from the fact that $E(X)$ is topologically nilpotent.

\begin{lem}
  Let $(X,T)$ be a distal minimal topological dynamical system. Then the center of $E(X)$ is the group of elements of $E(X)$ which are continuous.
 \end{lem}
 \begin{proof}
  Since $T$ commutes with every element of $E(X)$ it is clear that every continuous element of $E(X)$ belongs to the center of $E(X)$. Conversely, let 
  $u$ be in the center of $E(X)$ and $x_0 \in X$. We prove that $u$ is continuous at $x_0$. Suppose this is not true, and let $x_n\to x_0$ with $u(x_n)\to x'\neq u(x_0)$. By minimality, we can find $u_n \in E(X)$ such that $u_n(x_0)=x_n$. For a subsequence we have that $u_n\to v$ and $v(x_0)=x_0$. Since $u$ is central, we have $u(x_n)=u(u_n(x_0))=u_n(u(x_0))\to v(u(x_0))=u(v(x_0))=u(x_0)$, a contradiction.
 \end{proof}

Recall the following classical theorem:
  
  \begin{thm}(See \cite{Aus}, Chapter 4) \label{AusUnif}
   Let $G$ be a group of homeomorphisms of a compact Hausdorff space $X$ and suppose that $G$ is compact in the pointwise topology. Then, the action of $G$ on $X$ is equicontinuous. 
  \end{thm}

A direct consequence is:

\begin{cor} \label{compactness}
Let $(X,T)$ be distal topological dynamical system. If $E(X)$ is $d$-step topologically nilpotent, then $E_{d}^{\text{top}}(X)$ is a compact group of automorphisms of $(X,T)$ in the uniform topology.
\end{cor}

\begin{proof}
 If $E(X)$ is $d$-step topologically nilpotent, then $E_d^{\text{top}}(X)$ is a compact group (in the pointwise topology) and by definition is included in the center of $E(X)$, meaning that every element is an automorphism of $(X,T)$. By Theorem \ref{AusUnif}, $E_d^{\text{top}}(X)$ is a compact group of automorphisms in the uniform topology.

\end{proof}

If a system has a 2-step topologically nilpotent enveloping semigroup we can describe the extension of its maximal equicontinuous factor.

\begin{lem}  \label{E2Z1}
Let $(X,T)$ be a topological dynamical system with a 2-step topologically nilpotent enveloping semigroup. Then, it is an extension of $Z_1(X)$ by the compact abelian group $E_2^{\text{top}}(X)$. Moreover, $E_2^{\text{top}}(X)$ is connected.
\end{lem}
\begin{proof}
 By Corollary \ref{compactness} we have that $E_2^{\text{top}}(X)$ is a compact group of automorphisms of $(X,T)$ and by Lemma \ref{cubes} it acts trivially in every equicontinuous factor, meaning that there exists a factor map from $Z=X\backslash E_2^{\text{top}}(X) $ to $Z_1(X)$. Denote by $\pi$ the factor map from $X$ to $Z$ and note that if $u\in E_2^{\text{top}}(X)$, then $\pi(x)=\pi(ux)=\pi^{\ast}(u)\pi(x)$ for every $x\in X$ and therefore $\pi^{\ast}(u)$ is trivial. Since $e=\pi^{\ast}(E_2^{\text{top}}(X))=E_2^{\text{top}}(Z)$, we conclude that $Z$ has an abelian enveloping semigroup and thus it is an equicontinuous factor. By maximality $Z_1(X)=X\backslash E_2^{\text{top}}(X)$.

 If $E_2^{\text{top}}(X)$ is not connected, there exists an open (hence closed) subgroup $U\subseteq E_2^{\text{top}}(X)$ such that $E_2^{\text{top}}(X)/U$ is isomorphic to $\Z/n\Z$ for some $n>1$. Note that $X\backslash U $ is a finite-to-one extension of $Z_1(X)$. But a distal, finite-to-one extension of an equicontinuous system is an equicontinuous system (see \cite{Sacker-Sell} and for completeness we provide a short proof in our case in the appendix). By maximality we get that $X\backslash U =Z_1(X)$, a contradiction. 
\end{proof}
 
 This lemma establish the implication $(2)\Rightarrow (3)$ of Theorem \ref{thm2}. A direct corollary is:
   
   \begin{cor}
    Let $(X,T)$ be a system of order 2. Then, it is an extension of its maximal equicontinuous factor by the compact connected abelian group $E_2^{\text{top}}(X)$.
    
    \end{cor}
 
 \medskip

We now prove the main implication in Theorem \ref{thm2}, namely implication $(2)\Rightarrow (1)$.

\begin{proof}[Proof of implication $(2)\Rightarrow (1)$ ]

We divide the proof into four parts. The first two parts follow, with some simplifications, the scheme proposed in \cite{Glas}, but the second two parts are new.

\medskip

{\bf Step 1: Building a suitable extension of $(X,T)$.}
   
 Let $(X,T)$ be a topological dynamical system with a 2-step topologically nilpotent enveloping semigroup. By Lemma \ref{E2Z1}, $(X,T)$ is an extension of $(Z_1(X),T)$ by the compact abelian group $E_2^{\text{top}}(X)$. We denote this factor map by $\pi$. In order to avoid confusions, we denote the element of $Z_1(X)$ defining the dynamics by $\tau$ (instead of $T$). Let $\widehat{Z_1}$ be the dual group of $Z_1(X)$. Since $\{\tau^n:n\in \Z\}$ is dense in $Z_1(X)$ every $\chi\in \widehat{Z_1}$ is completely determined by its value at $\tau$ and thus we can identify $\widehat{Z_1}$ with a discrete subgroup of $\mathbb{S}^{1}$. Consider $\widehat{Z^{\ast}}=\{\lambda \in \mathbb{S}^1: \exists n\in \N,  \lambda^n \in \widehat{Z_1}\}$, the divisible group generated by $\widehat{Z_1}$. It is a discrete subgroup of $\mathbb{S}^1$, and we can consider its compact dual group $Z^{\ast}=\widehat{\widehat{Z^{\ast}}}$. Since $\widehat{Z^{\ast}}$ is a subgroup of the circle, $Z^{\ast}$ is a monothetic group with generator the identity character $\tau^{\ast}$: $\widehat{Z^{\ast}}\to \mathbb{S}$. Since $\widehat{Z_1}\subseteq \widehat{Z^{\ast}}$, there exists a homomorphism $\phi:Z^{\ast}\to Z_1(X)$. Since $\tau^{\ast}$ is projected to $\tau$, $\phi$ also defines a factor map from $(Z^{\ast},\tau^{\ast})$ to $(Z_1(X),\tau)$.  Consider $(X^{\ast},T\times \tau^{\ast})$ a minimal subsystem of $$(\{(x,z^{\ast})\in X\times Z^{\ast}:\pi(x)=\phi(z^{\ast})\},T\times \tau^{\ast}).$$ It is an extension of $(X,T)$ and $(Z^{\ast},\tau^{\ast})$ and we can see $E(X^{\ast})$ as a subset of $E(X)\times E(Z^{\ast})=E(X)\times Z^{\ast}$. It follows that $E_2^{\text{top}}(X^{\ast})=E_2^{\text{top}}(X)\times\{e\}$ and $E(X^{\ast})$ is 2-step topologically nilpotent. By Lemma \ref{E2Z1} we have that $Z_1(X^{\ast})=X^{\ast} \backslash E_2^{\text{top}}(X^{\ast}) =Z^{\ast}$. 
 
 \medskip
 
 {\bf Step 2: Finding a transitive group in $X^{\ast}$.}
 
 For simplicity we denote the transformation on $X^{\ast}$ by $T^{\ast}$. Let $(x_0,x_1)\in X^{\ast}\times X^{\ast}$. We construct a homeomorphism $h$ of $X^{\ast}$ such that $h(x_0)=x_1$. For this, define $Y$ as the closed orbit of $(x_0,x_1)$ under $T^{\ast}\times T^{\ast}$. Since $(X^{\ast},T^{\ast})$ is distal, $(Y,T^{\ast}\times T^{\ast})$ is a minimal distal system and $E(Y)=E(X^{\ast})^{\triangle}:=\{(u,u):u\in E(X^{\ast})\}$ (and we can identify $E(X^{\ast})$ and $E(Y)$). It follows that $E(Y)$ is 2-step topologically nilpotent and by Lemma \ref{E2Z1} $Z_1(Y)= Y\backslash E_2^{\text{top}}(Y) =Y \backslash E_2^{\text{top}}(X^{\ast})^{\triangle} $.
  
  We obtain the following commutative diagram:
   \begin{displaymath}
    \xymatrix{(Y,T^{\ast}\times T^{\ast}) \ar[d]_{p_{Y}} \ar[rr]_{\pi_{1}} &  & (X^{\ast},T^{\ast}) \ar[d]_{p_{X^{\ast}}}  \\ (Z_{1}(Y),{\tau_Y}^{\ast}) \ar[rr]_{\rho} & & (Z^{\ast},\tau^{\ast})
} 
\end{displaymath}

\medskip

Since $Z^{\ast}$ has a divisible dual group, we can identify $Z_1(Y)$ as a product group $Z^{\ast}\times G_0$ and we can write $p_Y(x,x')=(p_{X^{\ast}}(x),\Theta(x,x'))$ with $\Theta(x,x')\in G_0$. Since $Z_1(Y)$ is a product group, there exists $g_0\in G_0$ such that $\tau_Y^{\ast}=\tau^{\ast}\times g_0$. We remark that if $(x,x')$ and $(x,x'')\in Y$ then $(x,x')=u(x,x'')$ for some $u\in E(X^{\ast})$. Writing $x'=vx$ for $v\in E(X^{\ast})$ we deduce that $x''=[u,v]x'$. From this, we deduce that $G_0=\{ \id \times u : u \in E_2^{\text{top}}(X^{\ast}), (\id\times u) Y = Y\}$.  

For $x\in X^{\ast}$, define $h(x)$ as the unique element in $X^{\ast}$ such that 
\begin{equation}
 (x,h(x))\in Y \text{ and } p_Y(x,h(x))=(p_{X^{\ast}}(x),e). \label{2nil}
\end{equation}
By multiplying the second coordinate by a constant, we can suppose that $p_Y(x_0,x_1)=(p_{X^{\ast}}(x),e)$ and thus $h(x_0)=x_1$. 

\medskip

{\it Claim 1: $h$ is a homeomorphism of $X^{\ast}$}:
 \begin{itemize}
\item If $x_n \to x\in X^{\ast}$, then $p_Y(x_n,h(x_n))=(p_{X^{\ast}}(x_n),e)\to (p_{X^{\ast}}(x),e)=p_{Y}(x,h(x))$ and $h$ is continuous (if $p_{Y}(x,x')=p_{Y}(x,x'')$ then $x'=x''$).

\item If $h(x)=h(x')$ then $(x,h(x))$ and $(x',h(x))$ belong to $Y$ and then $x'=ux$ for $u\in E_2^{\text{top}}(X^{\ast})$. We have that $p_{Y}(x',h(x))=(p_{X^{\ast}}(x),e)=p_{Y}(x,h(x))$ and thus $x=x'$.
\item If $x'\in X^{\ast}$, we can find $x\in X$ such that $(x,x')\in Y$ and $p_{Y}(x,x')=(p_{X^{\ast}}(x),\Theta(x,x'))$. It follows that $p_{Y}(x,x')=(\id\times \Theta(x,x'))(p_{X^{\ast}}(x),e)$ and $p_{Y}(\Theta^{-1}(x,x')x,x')=(p_{X^{\ast}}(\Theta^{-1}(x,x')x),e)$. By definition $h(\Theta^{-1}(x,x')x)=x'$ and therefore $h$ is onto. This proves the claim.
\end{itemize}

\medskip

{\it Claim 2: $h$ commutes with $E_2^{\text{top}}(X^{\ast})$:}

For $u\in E_{2}^{\text{top}}(X^{\ast})$ we have that $p_{Y}(ux,uh(x))=p_{Y}(x,h(x))=(p_{X^{\ast}}(x),e)=(p_{X^{\ast}}(ux),e)=p_Y(ux,h(ux))$ and we deduce that $h$ commutes with $E_{2}^{\text{top}}(X^{\ast})$.

\medskip
  
{\it Claim 3: $[h,T^{\ast}]=g_0\in E_2^{\text{top}}(X^{\ast})$:}

 By a simple computation we have that 
$p_{Y}(T^{\ast}x,T^{\ast}h(x))=\tau_Y^{\ast}(p_{Y}(x,h(x)))$  $ =\tau_Y^{\ast}(p_{X^{\ast}}(x),e)=(p_{X^{\ast}}(Tx),g_0)=(p_{X^{\ast}}(T^{\ast}x),g_0h(T^{\ast}x))$ and $T^{\ast}h=g_0hT^{\ast}$. This proves the claim.

 \medskip
Define $G$ as the group of homeomorphisms $h$ of $X^{\ast}$ such that  
\begin{equation} \label{Nil2}
 [h,T^{\ast}]\in E_2^{\text{top}}(X^{\ast}) \text{ and } h \text{ commutes with } E_2^{\text{top}}(X^{\ast}). 
\end{equation}

\medskip

Then, for every pair of points in $X^{\ast}\times X^{\ast}$ we can consider a homeomorphism $h$ as in \eqref{2nil} and this transformation belongs to $G$. Thus $G$ is a group acting transitively on $X^{\ast}$. 

Let $\Gamma$ be the stabilizer of a point $x_0\in X^{\ast}$. We can identity (as sets) $X^{\ast}$ and $G/\Gamma$.

\medskip

{\bf Step 3: The application $g\to gx_0$ is open}

{\it Claim 4: There exists a group homomorphism $p:G\to Z^{\ast}$ such that $p(g)p_{X^{\ast}}(x)=p_{X^{\ast}}(gx)$. }

Since $E_2^{\text{top}}(X^{\ast})$ is central in $G$, we have $gp_{X^{\ast}}(x)=gE_2^{\text{top}}(X^{\ast})x=E_2^{\text{top}}(X^{\ast})gx$ and so the action of $g\in G$ can descend to an action $p(g)$ in $X^{\ast}\backslash E_2^{\text{top}}(X^{\ast})=Z^{\ast}$. By definition this action satisfies $p(g)p_{X^{\ast}}(x)=p_{X^{\ast}}(gx)$ and $p(T^{\ast})=\tau^{\ast}$. From this, we can see that $p(g)$ commutes with $\tau^{\ast}$ and thus $p(g)$ belongs to $Z^{\ast}$. Particularly, if  $h_1,h_2\in G$, we have that $p([h_1,h_2])$ is trivial and then $[h_1,h_1]x_0=ux_0$ for some $u\in E_2^{\text{top}}(X^{\ast})$. By \eqref{Nil2}, $[h_1,h_2]$ commutes with $T^{\ast}$ and thus $[h_1,h_2]$ coincides with $u$ in every point. Therefore $[G,G]\subseteq E_2^{\text{top}}(X^{\ast})$ and thus $G$ is a 2-step nilpotent Polish group. Since $G$ is transitive in $X^{\ast}$ we can check that $p$ is an onto continuous group homomorphism. This proves the claim.

 Since $G$ and $Z^{\ast}$ are Polish groups and $p$ is onto, we have that $p$ is an open map and the topology of $Z^{\ast}$ coincides with the quotient topology of $G/\text{Ker}(p)=G/E_2^{\text{top}}(X^{\ast})\Gamma$ (see \cite{BecKech}, Chapter 1, Theorem 1.2.6).

Now we prove that the map $g\to gx_0$ is open. Consider a sequence $g_n\in G$ such that $g_n x_0$ is convergent in $X^{\ast}$. Projecting to $Z^{\ast}$ we have that $p(g_n)p_{X^{\ast}}(x_0)$ is convergent and taking a subsequence we can assume that $p(g_n)$ is convergent in $Z^{\ast}$. Since $p$ is open, we can find a convergent sequence $h_n\in G$ such that $p(g_n)=p(h_n)$. This implies that $p_{X^{\ast}}(g_nx_0)=p_{X^{\ast}}(h_nx_0)$ and therefore there exists $u_n \in E_2^{\text{top}}(X^{\ast})$ such that $g_nx_0=u_nh_nx_0$. By the compactness of $E_2^{\text{top}}(X^{\ast})$ we can assume that $u_n$ is convergent and $u_nh_n$ is convergent too. This proves that the map is open. 

\medskip

{\bf Step 4: Cubes of order 3 in $X^{\ast}$ are completed in a unique way.}

 Let consider a sequence $\vec{n}_i=(n_i,m_i,p_i)\in \Z^3$ such that ${T^{\ast}}^{\vec{n}_i \cdot \e} x_0\to x_0 $ for every $\e\neq \vec{1}$. We prove that ${T^{\ast}}^{\vec{n}_i\cdot \vec{1}} x_0 \to x_0$. We see every transformation ${T^{\ast}}^{\vec{n}_i \cdot \e}$ as an element of $G$. Since the application $g\to gx_0$ is open, taking a subsequence, we can find $h_i,h_i',h_i''$ in $G$, converging to $h,h',h'' \in G$ such that ${T^{\ast}}^{n_i}x_0=h_ix_0$, ${T^{\ast}}^{m_i}x_0=h_i'x_0$ and ${T^{\ast}}^{p_i}x_0=h_i''x_0$. 

We have that 
\begin{align*}
 {T^{\ast}}^{n_i+m_i}x_0 & ={T^{\ast}}^{n_i}h_i' x_0=[{T^{\ast}}^{n_i},h_i']h_i' h_i x_0 \\
 {T^{\ast}}^{n_i+p_i}x_0 & ={T^{\ast}}^{n_i}h_i'' x_0=[{T^{\ast}}^{n_i},h_i'']h_i'' h_i x_0\\
 {T^{\ast}}^{m_i+p_i}x_0 & ={T^{\ast}}^{m_i}h_i'' x_0=[{T^{\ast}}^{m_i},h_i'']h_i'' h_i' x_0 \\
 {T^{\ast}}^{n_i+m_i+p_i}x_0&=[{T^{\ast}}^{m_i},h_i''][{T^{\ast}}^{n_i},h_i''][{T^{\ast}}^{n_i},h_i']h_i'' h_i' h_i x_0.
\end{align*}

Since $[G,G]$ is included in $E_2^{\text{top}}(X^{\ast})$, by taking a subsequence we can assume that $[{T^{\ast}}^{n_i},h_i']\to g_1$, $[{T^{\ast}}^{n_i},h_i'']\to g_2$ and $[{T^{\ast}}^{m_i},h_i'']\to g_3$ and these limits belong to $E_2^{\text{top}}(X^{\ast})$. Taking limits we conclude that $g_1x_0=g_2x_0=g_3x_0=x_0$ and since these transformations commute with ${T^{\ast}}$, we have that they are trivial.

We conclude that $\lim\limits_{i\to\infty} {T^{\ast}}^{n_i+m_i+p_i}x_0=x_0$ and thus $(x_0,x_0,x_0,x_0,x_0,x_0,x_0)\in (X^{\ast})^7$ can be completed in a unique way to an element of $\Q^{[3]}(X^{\ast})$. If $\pi_2$ is the factor map from $X^{\ast}$ to $Z_2(X^{\ast})$, we have that $\#\pi_2^{-1}(\pi_2(x_0))=1$ and since $(X^{\ast},{T^{\ast}})$ is distal, the same property holds for every element in $X^{\ast}$. We conclude that $X^{\ast}=Z_2(X^{\ast})$ and thus $(X^{\ast},T^{\ast})$ is a system of order 2.

Since being a system of order 2 is a property preserved under factor maps, $(X,T)$ is a system of order $2$.
 
\end{proof}

We have established $(1)\Leftrightarrow (2)$ and $ (2) \Rightarrow (3)$. Since implication $(3) \Rightarrow (4)$ is obvious, we only have to prove $(4)\Rightarrow (2)$.

For this, we first prove the following lemma:

\begin{lem} \label{Isogroup}
  Let $\pi:X\to Y$ be an isometric extension between the minimal distal systems $(X,T_X)$ and $(Y,T_Y)$. Then, there exists a minimal distal system $(W,T_W)$ which is a group extension of $X$ and a group extension of $Y$. If $E(X)$ is $d$-step nilpotent then $E(W)$ is also $d$-step nilpotent. 
 \end{lem} 
 
  \begin{proof}

 Fix $y_0 \in Y$ and let $F_0=\pi^{-1}(y_0)$. Define $$\widetilde{Z}=\{(y,h): y \in Y , \ h\in \text{Isom}(F_0 , \pi^{-1}(y)) \} $$ 
 It is compact metrizable space and we can define $T_{\widetilde{Z}}:\widetilde{Z}\to \widetilde{Z}$ as $T_{\widetilde{Z}}(y,g)=(T_Y(y), T_X\circ h) $. We remark that $(\widetilde{Z},T_{\widetilde{Z}})$ is a distal system and we can see $E(\widetilde{Z})$ as a subset of $E(Y)\times E(X)$. It follows that $E(\widetilde{Z})$ is $d$-step nilpotent.
 
 Let $H$ denote the compact group of isometries of $F_0$ which are restrictions of elements of $E(X)$. We define the action of $H$ on $\widetilde{Z}$ as $g (y,h)=(y,h\circ g^{-1})$ and we define the maps $\pi_Y(y,h)=y$ and $\pi_X(y,h)=h(x_0)$ from $\widetilde{Z}$ to $X$ and $Y$. Define $W$ as the orbit of $(y_0,\id)$ under $T_{\widetilde{Z}}$ and let $T_W$ denote the restriction of $T_{\widetilde{Z}}$ to $W$. Since $(\widetilde{Z},T_{\widetilde{Z}})$ is a distal system, $(W,T_W)$ is a minimal system and therefore the restrictions of $\pi_Y$ and $\pi_X$ define factor maps from $(W,T_W)$ to $(Y,T_Y)$ and $(X,T_X)$. Since $(X,T_X)$ is a distal system we have that $(E(X),T_X)$ is a minimal system and we have that $\{y_0\}\times H\subseteq W$. We conclude that $(W,T_W)$ is an extension of $(Y,T_Y)$ by the group $H$ and thus it is an extension of $(X,T)$ by the group $H_0=\{h\in H: h(x_0)=x_0\}$. Since $(W,T_W)$ is a subsystem of $(\widetilde{Z},T_{\widetilde{Z}})$, we also have that $E(W)$ is $d$-step nilpotent. The lemma is proved.
  \end{proof}

Now we prove the implication $(4)\Rightarrow (2)$.

\begin{proof}[ Proof of implication $(4)\Rightarrow (2)$]
 Let $(X,T)$ be system with a 2-step nilpotent enveloping semigroup and let $\pi:X\to Y$ be an isometric extension of the equicontinuous system $(Y,T)$. By Lemma \ref{Isogroup} we can find $(W,T)$ which is an extension of $(Y,T)$ by a group $H$ and such that $E(W)$ is a 2-step nilpotent group. By Lemma \ref{cubes}, for $u\in E_2(W)$ and $w\in W$ we have that $uw$ and $w$ have the same projection on $Y$ and therefore there exists $h_w\in H$ such that $uw=h_ww$. Since $u$ and $h_w$ are automorphisms of the minimal system $(W,T)$, they are equal. Thus we have that $E_2(W)$ is a subgroup of $H$ and therefore $E_2^{\text{top}}(W)$ is just the closure (pointwise or uniform) of $E_2(W)$. We conclude that $E_2^{\text{top}}(W)$ is included in $H$ and therefore is central in $E(W)$. Since $(X,T)$ is a factor of $(W,T)$, $E_2^{\text{top}}(X)$ is also central in $E(X)$. This finishes the proof.

\end{proof}

\section{Some further comments }

We finish with some remarks about the structure of systems having topologically nilpotent enveloping semigroups. 

Let $(X,T)$ be a topological dynamical system and let $d>2$. Suppose that $E(X)$ is a $d$-step topologically nilpotent group. By Corollary \ref{compactness} $E_d^{\text{top}}(X)$ is a compact group of automorphisms of $(X,T)$ and thus we can build the quotient $X_{d-1}=X\backslash E_d^{\text{top}}(X)$. 

\begin{lem} \label{NilMax}
 $X_{d-1}$ has a $(d-1)$-step topologically nilpotent enveloping semigroup. Moreover it is the maximal factor of $X$ with this property and consequently $(X_{d-1},T)$ is an extension of $(Z_{d-1}(X),T)$ 
\end{lem}

\begin{proof}
 Denote by $\pi:X\to X_{d-1}$ the quotient map. If $u\in E_d^{\text{top}}(X)$, by definition we have that $\pi(x)=\pi(ux)=\pi^{\ast}(u)(x)$ and thus $\pi^{\ast}(u)$ is trivial. Since $\pi^{\ast}(E_d^{\text{top}}(X))=E_d^{\text{top}}(X_{d-1})$ we have that $E_d^{\text{top}}(X_{d-1})$ is trivial. 
 
 Let $(Z,T)$ be a topological dynamical system with a $(d-1)$-step topologically nilpotent enveloping semigroup and let $\phi:X\to Z$ be a factor map. Since $\phi^{\ast}(E_d^{\text{top}}(X))=e$, for $u\in E_d^{\text{top}}(X)$ we have that $\phi(ux)=\phi^{\ast}(u)\phi(x)=\phi(x)$ and therefore $\phi$ can be factorized through $X_{d-1}$.

 As $Z_{d-1}(X)$ has a $(d-1)$-step enveloping semigroup, we have that $(X_{d-1},T)$ is an extension of $(Z_{d-1}(X),T)$. 
\end{proof}

Iteratively applying Lemma \ref{NilMax}, we construct a sequence of factors $X_j$, for $j\leq d-1$ with the property that $X_j$ is an extension of $Z_j(X)$ and is an extension of $X_{j-1}$ by the compact abelian group $E_j^{\text{top}}(X_j)$.

By Theorem \ref{thm2}, the factors $X_2$ and $Z_2(X)$ coincide and we obtain the following commutative diagram:
  
\begin{displaymath}
    \xymatrix{(X,T) \ar[r] \ar[d] & (X_{d-1},T)  \ar[d]  \ar[r] & \cdots   \ar[r]& (X_3,T) \ar[dr] \ar[d]  \ar[d] &   \\
    (Z_d(X),T)\ar[r] & (Z_{d-1}(X),T) \ar[r] \ar[r] & \cdots \ar[r] & (Z_3(X),T) \ar[r] & (Z_2(X),T) \ar[r] & (Z_1(X,T)
} 
\end{displaymath}

We conjecture that the factor $X_j$ and $Z_j(X)$ also coincide for $j>2$.

\appendix

\section{Finite-to-one extension of equicontinuous systems}

For completeness, in this appendix we give a short proof of \cite{Sacker-Sell} in our context.
  
  \begin{thm}
   Let $\pi:X\to Y$ be a distal finite-to-one factor map between the minimal systems $(X,T)$ and $(Y,T)$. Then $(Y,T)$ is equicontinuous if and only if $(X,T)$ is equicontinuous.     
  \end{thm}

\begin{proof} We prove the non trivial direction by studying the regionally proximal relation on $X$. We denote by $d_X$ and $d_Y$ the metrics on $X$ and $Y$. We can assume that $T$ is an isometry on $Y$. Since $\pi$ is open and finite-to-one, there exists $\e_0>0$ such that for every $y\in Y$ every ball of radius $2\e_0$ in $X$ intersects $\pi^{-1}(y)$ in at most one point. Let $\e_1<\e_0$ such that $T(B(x,\e_1))\subseteq B(Tx,\e_0)$. Since $\pi$ is open, there exists $\delta>0$ with the property that if $y,y' \in Y$ are such that  $d_Y(y,y')<\delta$ then there exists $x,x'\in X$ with $d_X(x,x')<\e_1$ and $\pi(x)=y$, $\pi(x')=y'$. Let $0<\e<\e_1$ such that $\pi(B(x,\e))\subseteq B(\pi(x),\delta)$. Let $(x,x')$ be a regionally proximal pair, and let $x''\in X$  and $n_0 \in \N$ satisfying $d_X(x,x'')<\e$ and $d_X(T^{n_0}x',T^{n_0}x'')<\e$. We have that $d_Y(T^n \pi(x),T^n \pi(x''))=d_Y(\pi(T^nx),\pi(T^nx''))<\delta$ for every $n\in \N$ and by openness, we can find $x_n\in X$ such that $\pi(x_n)=\pi(T^nx)$ and $d_X(x_n,T^n x'')<\e_1$.

We claim that $x_n=T^nx$. We proceed by induction. For $n=0$ we have $d_X(x_0,x)< 2 \e_0$, $\pi(x)=\pi(x_0)$ and thus $x=x_0$. Suppose now that $x_n=T^nx$. We have that $d_X(T^nx,T^nx'')<\e_1$ and then $d_X(T^{n+1}x,T^{n+1}x'')<\e_0$. We conclude that $d_X(x_{n+1},T^{n+1}x)< 2\e_0$, and since they have the same projection, they are equal. This proves the claim.

Particularly, for $n=n_0$, we have $d_X(T^{n_0}x,T^{n_0}x')<2\e_0$ and since they are regionally proximal, they have the same projection and thus $x=x'$. We conclude that the regionally proximal relation is trivial and $(X,T)$ is equicontinuous.  \end{proof}

\end{document}